\documentclass[12pt]{amsart}

\usepackage{amssymb,latexsym,amsmath,amscd}
\usepackage{graphicx,epstopdf}

\usepackage{fullpage}
\usepackage{url,hyperref,cleveref}

\newtheorem*{theorem*}{Theorem}
\newtheorem{theorem}{Theorem}[section]
\newtheorem*{lemma*}{Lemma}
\newtheorem{lemma}[theorem]{Lemma}
\newtheorem*{corollary*}{Corollary}
\newtheorem{corollary}[theorem]{Corollary}
\newtheorem*{conjecture*}{Conjecture}

\theoremstyle{definition}
\newtheorem{definition}[theorem]{Definition}
\newtheorem{example}[theorem]{Example}

\newcommand{\inner}[2]{\left\langle #1,#2 \right\rangle}
\newcommand{\til}{\widetilde}
\newcommand\C{\mathbb{C}}

\newcommand\Z{\mathbb{Z}}
\newcommand\R{\mathbb{R}}
\newcommand\Aut{\mathrm{Aut}}

\newcommand\inv{^{-1}}
\newcommand\id{\operatorname{id}}
\newcommand{\mx}[1]{\mathbf{#1}}

\title[IMGs of Chebyshev-like maps]
{Iterated monodromy groups of \\ Chebyshev-like maps on $\C^n$}
\author[J. Bowman]{Joshua P. Bowman}
\address{Natural Science Division, Seaver College, Pepperdine University}
\email{joshua.bowman@pepperdine.edu}
\date{\today}

\begin{document}

\begin{abstract}
Every affine Weyl group appears as the iterated monodromy 
group of a Chebyshev-like polynomial self-map of $\C^n$.
\end{abstract}

\maketitle

%

\section*{Introduction}

The Chebyshev polynomials $T_d : \C \to \C$, defined for 
$d \ge 2$ by the equation $T_d(\cos\theta) = \cos d\theta$, 
are important in many areas of mathematics. In the study of 
single-variable complex dynamics, they are especially 
notable for being post-critically finite, and for having 
a ``smooth'' Julia set, namely the interval $[-1,1]$. 
Moreover, their restrictions to $[-1,1]$ act as ``folding 
maps,'' whose dynamics can be completely described. 

In the 1980s, Veselov \cite{aV86} and Hoffman--Withers 
\cite{mHwW88} defined, for each root system $\Phi$ in 
$\R^n$, a family of ``Chebyshev-like'' polynomial 
maps $T_{\Phi,d} : \C^n \to \C^n$. (Certain special cases, 
particularly in dimension $2$, had been studied since the 
1970s \cite{tK74,kDrL82}, although not in a dynamical 
context.) These maps $T_{\Phi,d}$ are also post-critically 
finite (in an appropriate sense), and each of them acts as 
a ``folding map'' on a certain compact subset of $\C^n$, 
which depends only on $\Phi$.

Any post-critically finite map $f : \C^n \to \C^n$ has 
an associated iterated monodromy group $\mathrm{IMG}(f)$, 
which encodes the dynamics of $f$ algebraically, especially 
on (the boundary of) the set of points in $\C^n$ that 
do not escape to infinity under iteration of $f$. 
Iterated monodromy groups were introduced by Nekrashevych 
\cite{vN-img,vN-ssg,vN-sym} and have proved to be a powerful 
tool in both dynamics and group theory. However, very few 
such groups have been calculated for post-critically finite 
maps of $\C^n$ where $n > 1$; see \cite{jBsK10,vN-julia} 
for the only examples known to the author. (A special case 
of the present article's main result, obtained by 
different methods, is shown in \cite{jB21}.)

It is known that the iterated monodromy group of a 
Chebyshev polynomial $T_d$ is the infinite dihedral group 
$\langle a,b \mid a^2 = b^2 = \id \rangle$, which may be 
realized as the group of transformations of $\theta\in\C$ 
that leave the cosine $\frac12(e^{i\theta}+e^{-i\theta})$ 
invariant. In this article, we generalize this result to 
Chebyshev-like maps in every dimension $n \ge 1$. Given a root
system $\Phi \subset \mathbb{R}^n$, we let $\til{W}_\Phi$ 
denote the associated affine Weyl group.

\begin{theorem*}
Let $T_{\Phi,d} : \C^n \to \C^n$ be a Chebyshev-like 
map associated to the root system $\Phi$. Then 
$\mathrm{IMG}(T_{\Phi,d})$ is isomorphic to $\til{W}_\Phi$.
\end{theorem*}

Our approach is somewhat indirect. Because $T_{\Phi,d}$ 
is post-critically finite, its post-critical locus 
is a (not necessarily irreducible) hypersurface 
$\mathcal{D}_\Phi \subset \C^n$, which we show depends 
only on $\Phi$. The iterated monodromy group 
$\mathrm{IMG}(T_{\Phi,d})$ is defined as a quotient of 
the fundamental group of the complement of $\mathcal{D}_\Phi$. 
However, we do not compute this fundamental group at the 
start; instead, we relate $\mathrm{IMG}(T_{\Phi,d})$ 
and $\pi_1(\C^n\setminus\mathcal{D}_\Phi)$ to the 
fundamental group of a certain complement of hyperplanes 
(the ``Cartan--Stiefel diagram,'' see \S\ref{S:critical}). 
Along the way, we uncover 
$\pi_1(\C^n\setminus\mathcal{D}_\Phi)$.

\begin{corollary*}
For every root system $\Phi$, there exists an infinite 
hyperplane arrangement $\mathcal{H}_\Phi\subset\C^n$, 
invariant under $\til{W}_\Phi$, such that 
$\pi_1(\C^n \setminus \mathcal{D}_\Phi)$ 
is isomorphic to an extension of $\til{W}_\Phi$ 
by $\pi_1(\C^n \setminus \mathcal{H}_\Phi)$.
\end{corollary*}

Iterated monodromy groups are examples of 
self-similar groups acting on trees. Thus 
we also have the following consequence.

\begin{corollary*}
Any affine Weyl group of rank $n$ acts faithfully 
as a self-similar group on a rooted $d^n$-ary tree 
for any $d \ge 2$.
\end{corollary*}

It is natural to conjecture that the property of having 
an iterated monodromy group isomorphic to an affine 
Weyl group characterizes the Chebyshev-like maps.

\begin{conjecture*}
Let $f : \C^n \to \C^n$ be post-critically finite. 
If the iterated monodromy group of $f$ is an affine 
Weyl group, then (some iterate of) $f$ is a Chebyshev-like map.
\end{conjecture*}

Here is the structure of the paper.
In \S\ref{S:root}, we recall the necessary definitions 
from the theory of root systems and review the definition 
of the Chebyshev-like maps $T_{\Phi,d}$. 
In \S\ref{S:critical}, we study the 
post-critical locus of each map $T_{\Phi,d}$. 
In \S\ref{S:img}, we recall the definition of iterated 
monodromy groups and establish a key lemma. 
Finally, in \S\ref{S:computation} we prove the main result.

\section{Root systems and Chebyshev-like maps}\label{S:root}

First we review some of the theory of root systems. 
References are \cite{hS90,bH03}. The notation used here 
is similar but not identical to that of \cite{mHwW88}.

Throughout, we endow $\R^n$ with the standard 
inner product $\inner{\,\cdot\,}{\,\cdot\,}$, 
which we extend to a Hermitian form on $\C^n$, 
also written $\inner{\,\cdot\,}{\,\cdot\,}$, 
that is antilinear in the first variable and 
complex linear in the second variable---that is, 
for all $\lambda\in\C$ and $\mx{v},\mx{w}\in\C^n$, 
we have $\inner{\mx{v}}{\lambda\mx{w}} 
= \lambda \inner{\mx{v}}{\mx{w}} 
= \inner{\bar\lambda\mx{v}}{\mx{w}}$.

\begin{definition}[complex reflection]
A nonzero vector $\mx{v} \in \C^n$ and a real number 
$\ell \in \R$ together determine a {\em complex 
reflection} $\rho_{\mx{v},\ell} : \C^n \to \C^n$, 
given algebraically by 
\[
\rho_{\mx{v},\ell}(\mx{x}) 
= \mx{x} 
  - 2\frac{\inner{\mx{v}}{\mx{x}} - \ell}{\inner{\mx{v}}{\mx{v}}} \mx{v}.
\]
\end{definition}

Note that $\rho_{\mx{v},\ell}$ is complex-affine in $\mx{x}$, 
and its derivative is $D\rho_{\mx{v},\ell} = \rho_{\mx{v},0}$. 
The fixed-point set of $\rho_{\mx{v},\ell}$ is the complex 
hyperplane $H_{\mx{v},\ell}$ defined by the equation 
$\inner{\mx{v}}{\mx{x}} = \ell$. If $\mx{v} \in \R^n$, then 
$\rho_{\mx{v},\ell}$ restricts to an ordinary reflection 
$\R^n \to \R^n$ across the real hyperplane 
$H_{\mx{v},\ell} \cap \R^n$.

\begin{definition}[root system, root, coroot]
A {\em root system} (with \emph{rank $n$}) is a finite set 
of vectors $\Phi \subset \R^n$ such that the following 
conditions are satisfied:
\begin{itemize}
\item $\Phi$ spans $\R^n$;
\item if $\mx{v} \in \Phi$ and $\lambda \in \R$, then 
$\lambda\mx{v} \in \Phi \iff \lambda = \pm1$;
\item if $\mx{v} \in \Phi$, then 
$\rho_{\mx{v},0}(\mx{w}) \in \Phi$ 
for all $\mx{w} \in \Phi$;
\item if $\mx{v} \in \Phi$ and $\mx{w} \in \Phi$, then 
$2 \dfrac{\inner{\mx{v}}{\mx{w}}}{\inner{\mx{v}}{\mx{v}}} \in \Z$.
\end{itemize}
Elements of $\Phi$ are called {\em roots}. 
For $\mx{v} \in \Phi$, the {\em coroot} of $\mx{v}$ 
is $\mx{v}^\vee = 
\dfrac{2\,\mx{v}}{\inner{\mx{v}}{\mx{v}}}$.
\end{definition}

A root system $\Phi$ is {\em irreducible} if it cannot 
be partitioned into root systems of lower rank, which 
are contained in orthogonal subspaces. 

\begin{figure}[bth]
\centering
\begin{minipage}{1.2in}
\centering
\includegraphics[scale=1.25]{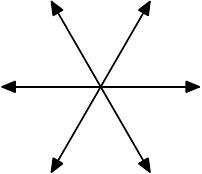}
\end{minipage}
\hspace{0.35in}
\begin{minipage}{1.1in}
\centering
\includegraphics[scale=1.1]{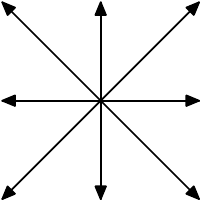}
\end{minipage}
\hspace{0.35in}
\begin{minipage}{1in}
\centering
\includegraphics{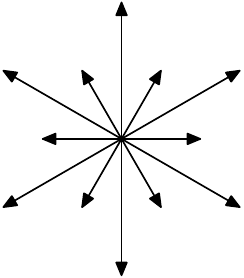}
\end{minipage}

\caption{Irreducible root systems of rank 2: 
respectively, $A_2$, $B_2$, and $G_2$.}\label{F:rank2}
\end{figure}

When the type of $\Phi$ is known 
(such as $A_n$, $B_n$, $C_n$, $D_n$, $E_6$, $E_7$, $E_8$, 
$F_4$, or $G_2$ in the cases that $\Phi$ is irreducible), 
then in notation we may replace $\Phi$ with its type.

\begin{definition}[Weyl group, affine Weyl group]
Given a root system $\Phi$, the {\em Weyl group} 
of $\Phi$ is the group $W_\Phi$ generated by all 
reflections of the form $\rho_{\mx{v},0}$, 
$\mx{v} \in \Phi$. The {\em affine Weyl group} of 
$\Phi$ is the group $\til{W}_\Phi$ generated by all 
reflections of the form $\rho_{\mx{v},\ell}$, 
$\mx{v} \in \Phi$, $\ell \in \Z$. 
\end{definition}

Equivalently, the affine Weyl group $\til{W}_\Phi$ 
can be defined as the semidirect product 
$Q_\Phi^\vee \rtimes W_\Phi$, where 
$Q_\Phi^\vee$ is the lattice in $\R^n$ generated by 
the coroots of $\Phi$. 
Both $W_\Phi$ and $\til{W}_\Phi$ may be thought of 
as acting on either $\R^n$ or $\C^n$.

\begin{definition}[simple roots, fundamental weights]
Given a root system $\Phi \subset \R^n$, let 
$\phi : \R^n \to \R$ be a linear functional that does 
not vanish on any elements of $\Phi$. The elements 
$\mx{v}$ of $\Phi$ such that $\phi(\mx{v}) > 0$ are 
called {\em positive roots} (relative to $\phi$); 
a positive root $\alpha$ is called {\em simple} if it 
cannot be written as a sum $\alpha = \mx{v} + \mx{w}$ 
where $\mx{v}$ and $\mx{w}$ are distinct positive roots. 
The simple roots form a basis $\{\alpha_1,\dots,\alpha_n\}$ 
of $\R^n$; the {\em fundamental weights} 
$\omega_1, \dots, \omega_n$ form the dual basis to 
the coroots of the simple roots: that is, 
\[
\inner{\omega_j}{\alpha_k^\vee} 
= \begin{cases}
  1 & \text{if}\ j = k\text, \\
  0 & \text{if}\ j \ne k\text. 
  \end{cases}
\]
\end{definition}

Now we reach our main definition.

\begin{definition}[generalized cosine, Chebyshev-like map]
Let $\Phi \subset \R^n$ be a root system and $W_\Phi$ 
its Weyl group. Let $\omega_1, \dots, \omega_n$ be 
a choice of fundamental weights for $\Phi$. For each 
$1 \le k \le n$, define $\psi_k$ from $\C^n$ to $\C$ by 
\begin{align*}
\psi_j(\mathbf{x}) 
:\!&= 
   \sum_{\mathbf{r} \in W_\Phi\omega_k} 
   \exp\!\big(2\pi i\inner{\mathbf{r}}{\mx{x}}\!\big) \\
&= \frac{1}{|\mathop{\mathrm{Stab}}_{W_\Phi}(\omega_k)|}
   \sum_{w \in W_\Phi} 
   \exp\!\big(2\pi i\inner{w\omega_k}{\mx{x}}\!\big),
\end{align*}
where $W_\Phi\omega_k$ is the orbit of $\omega_k$ under 
$W_\Phi$, and $\mathop{\mathrm{Stab}}_{W_\Phi}(\omega_k)$ 
is the stabilizer of $\omega_k$ in $W_\Phi$. 
Then define the {\em generalized cosine} 
$\Psi_\Phi : \C^n \to \C^n$ by 
\[
\Psi_\Phi := \big(\psi_1,\dots,\psi_n\big).
\]
For each integer $d \ge 2$, let $m_d : \C^n \to \C^n$ 
denote multiplication by $d$. Then the {\em Chebyshev-like map} 
$T_{\Phi,d} : \C^n \to \C^n$ is defined by the functional equation 
\begin{equation}\label{E:chebysheveq}
T_{\Phi,d} \circ \Psi_\Phi = \Psi_\Phi \circ m_d \text.
\end{equation}
\end{definition}

\begin{theorem*}[\cite{aV86,mHwW88}]
Given any root system $\Phi$ of rank $n$ and any 
integer $d \ge 2$, equation \eqref{E:chebysheveq} 
defines a polynomial map $T_{\Phi,d} : \C^n \to \C^n$.
\end{theorem*}

The construction above is due independently to 
Veselov \cite{aV86,aV91} and Hoffman--Withers \cite{mHwW88}. 
Up to permutation of coordinates, $\Psi_\Phi$ 
is independent of the choice of fundamental weights, 
because the Weyl group acts transitively on bases of 
simple roots. The terminology of ``generalized cosine'' 
comes from Hoffman and Withers, who described these 
Chebyshev-like maps in terms of folding figures in $\R^n$, 
which leads to the consideration of root systems. Veselov 
expressed the construction in terms of exponential 
invariants of semi-simple Lie algebras and noted, from 
Chevalley's theorem, that the coefficients of the 
polynomials defining $T_{\Phi,d}$ are integers.

\begin{example}[Chebyshev polynomials]\label{Ex:A1}
In the classical case, from which the Chebyshev-like 
maps get their name, $\Phi = \{\pm1\}$ is the $A_1$ 
root system in $\R$. Then $1$ is a simple root, $\frac{1}{2}$ 
is the corresponding fundamental weight, and the 
Weyl group is just the two-element group generated 
by multiplication by $-1$ (in either $\R$ or $\C)$. 
Set $t = e^{i\,\pi x}$, so that 
$\Psi_{A_1}(x) = \psi_1(x) = t + t\inv$, 
and we have the $d$th Chebyshev polynomial 
$T_d := T_{A_1,d}$ defined by the equation 
$T_d(t+t\inv) = t^d+t^{-d}$. 
(Here we have followed the convention that 
$T_d(2\cos\theta) = 2\cos d\theta$, contra the 
equation $T_d(\cos\theta) = \cos d\theta$ stated 
in the introduction. These two conventions produce
dynamically conjugate maps.)
\end{example}

\begin{example}[A Chebyshev-like map in $2$ dimensions]\label{Ex:A2}
The $A_2$ root system is the simplest of the irreducible 
rank 2 root systems (see Figure~\ref{F:rank2}). It can be 
realized in the plane in $\R^3$ having equation 
$x_1 + x_2 + x_3 = 0$ as the set of six vectors 
$\Phi = \{(\pm1,\mp1,0),(0,\pm1,\mp1),(\pm1,0,\mp1)\}$.
One choice of simple roots is $\alpha_1 = (1,-1,0)$, 
$\alpha_2 = (0,1,-1)$. 
The corresponding fundamental weights are 
$\omega_1 = (2/3,-1/3,-1/3)$ and 
$\omega_2 = (1/3,1/3,-2/3)$, 
and so when $\mx{x} = (x_1,x_2,x_3)$ satisfies 
$x_1 + x_2 + x_3 = 0$, 
we have $\inner{\omega_1}{\mx{x}} = x_1$ and 
$\inner{\omega_2}{\mx{x}} = x_1 + x_2$. The Weyl group in 
this case is the symmetric group on $3$ elements, realized 
as the permutations of the coordinates in $\R^3$. 
If we set $t_j = \exp(i\,2\pi x_j)$ and $(X_1,X_2) = 
\Psi_{A_2}(\mx{x})$, then we have the equalities 
$t_1 t_2 t_3 = 1$, $X_1 = t_1 + t_2 + t_3$ and 
$X_2 = t_1 t_2 + t_2 t_3 + t_3 t_1$. In the coordinates 
$(X_1,X_2)$, we may write $T_{A_2,2}$, for instance, 
as the map 
\[
T_{A_2,2}(X_1,X_2) = \big({X_1}^2 - 2 X_2, {X_2}^2 - 2 X_1\big),
\]
which has been independently studied, e.g., 
in \cite{jB21,sN08,kU09}.
\end{example}

\section{Critical and post-critical loci}\label{S:critical}

\begin{definition}[critical point, critical value, 
post-critical locus, post-critically finite]
Let $f : \C^n \to \C^n$ be a holomorphic map. 
A {\em critical point} of $f$ is a point $\mx{c}$ 
such that the derivative $Df(\mx{c}) : \C^n \to \C^n$ 
is singular. 
The {\em critical locus} of $f$ is the set 
$\mathcal{C}_f$ containing all critical points of $f$. 
A {\em critical value} of $f$ is a point of the form 
$f(\mx{c})$, where $\mx{c}\in\mathcal{C}_f$. 
The {\em post-critical locus} of $f$ is 
the union $\mathcal{P}_f$ of all (strict) forward 
images of the critical locus of $f$, in symbols
\[
\mathcal{P}_f 
:= \bigcup_{k\ge1} f^k(\mathcal{C}_f)\text.
\] 
We say $f$ is {\em post-critically finite} 
if $\mathcal{C}_f \ne \C^n$ and $\mathcal{P}_f$ 
is a closed, proper subvariety of $\C^n$.
\end{definition}

The notion of a post-critically finite map of 
$\C^n$ was introduced by Forn{\ae}ss and Sibony 
\cite{jeFnS92,jeFnS95} as a generalization of 
post-critically finite polynomials on $\C$.
The post-critical locus of such a map 
$f$ includes the critical values of $f$, but it may 
(and generally does) include more points of $\C^n$. 
The map $f$ is locally a covering map away from its 
critical values. Thus, the restriction of $f$ to the 
complement of $\mathcal{C}_f \cup \mathcal{P}_f$ is 
a covering of the complement of $\mathcal{P}_f$. 

Although we will not treat the generalized cosine 
$\Psi_\Phi$ dynamically, we do need to know what 
its critical locus and critical values are.

\begin{definition}[Cartan--Stiefel diagram]
Given a root system $\Phi$, let $\mathcal{H}_\Phi$ 
be the union of all complex hyperplanes fixed by some 
non-identity element of the affine Weyl group $\til{W}_\Phi$.
That is, 
\[
\mathcal{H}_\Phi 
:= \bigcup_{\mx{v}\in\Phi,\,\ell\in\Z} H_{\mx{v},\ell}\,\text.
\]
This union of hyperplanes is the (complex) 
\emph{Cartan--Stiefel diagram} of $\til{W}_\Phi$.
\end{definition}

We also let $\mathcal{D}_\Phi$ be the image of 
$\mathcal{H}_\Phi$ by $\Psi_\Phi$; that is,
\[
\mathcal{D}_\Phi := \Psi_\Phi(\mathcal{H}_\Phi)\text.
\]

\begin{example}
When $\Phi$ is the $A_1$ root system as in Example~\ref{Ex:A1}, 
we have $\mathcal{H}_\Phi = \Z$ and $\mathcal{D}_\Phi = \{\pm2\}$.
\end{example}

The next example provides part of the motivation for 
the notation $\mathcal{D}_\Phi$.

\begin{example}
When $\Phi$ is the $A_2$ root system as in Example~\ref{Ex:A2},
the points of $\mathcal{H}_\Phi$ with real coordinates 
form the edges of a planar tiling by equilateral triangles. 
$\mathcal{D}_\Phi$ is the complex version of the 
deltoid (a.k.a.\ three-cusped hypocycloid) with equation 
${X_1}^2 {X_2}^2 + 18 X_1 X_2 = 4 ({X_1}^3 + {X_2}^3) + 27$. 
\end{example}

The next three lemmas demonstrate the importance of 
$\mathcal{H}_\Phi$ and $\mathcal{D}_\Phi$ to our study.

\begin{lemma}\label{L:Psi-critical}
For any root system $\Phi$, the critical locus of 
$\Psi_\Phi$ is $\mathcal{H}_\Phi$.
\end{lemma}

\begin{proof}
It is evident from the definition of $\Psi_\Phi$ that 
$\Psi_\Phi(\til{w}\mx{x}) = \Psi_\Phi(\mx{x})$ for all 
$\til{w} \in \til{W}_\Phi$, and in particular that the 
set of critical points of $\Psi_\Phi$ is invariant under 
the action of $\til{W}_\Phi$. Thus it is sufficient to 
show that $\mx{c}$ is a critical point for $\Psi_\Phi$ 
if and only if it is equivalent under $\til{W}_\Phi$ to 
some point fixed by a non-identity element of $W_\Phi$.

First, note that if $\rho_{\mx{v},0}(\mx{c}) = \mx{c}$ 
for some $\mx{v} \in \Phi$, then for all $\lambda \in \C$ 
we have $\rho_{\mx{v},0}(\mx{c}+\lambda\mx{v}) = 
\mx{c}-\lambda\mx{v}$, and thus $\Psi_\Phi$ is not 
locally injective at $\mx{c}$; in other words, $\mx{c}$ 
is a critical point of $\Psi_\Phi$. 
Therefore, all of $\mathcal{H}_\Phi$ is contained in 
the critical locus of $\Psi_\Phi$.

To see that $\Psi_\Phi$ has no other critical points, 
first observe that the natural projection 
$\C^n \to \C^n/Q_\Phi^\vee$ is a covering map, having 
no critical points. Next, $\Psi_\Phi$ is the composition 
of this projection and the quotient map 
$\C^n/Q_\Phi^\vee \to \C^n$ given by the induced action 
of $W_\Phi$ on $\C^n/Q_\Phi^\vee$. The critical points 
of this latter action are precisely the points that 
are fixed by some non-identity element of $W_\Phi$, 
which is to say, the image of $\mathcal{H}_\Phi$ in 
$\C^n/Q_\Phi^\vee$.
\end{proof}

\begin{lemma}
Given a root system $\Phi$ and an integer $d \ge 2$, 
the Chebyshev-like map $T_{\Phi,d}$ is post-critically 
finite, with $\mathcal{D}_\Phi$ as its post-critical 
locus. 
\end{lemma}

\begin{proof}
Differentiating both sides of \eqref{E:chebysheveq} 
at a variable point $\mx{x}$ and applying the chain 
rule yields 
\[
[DT_{\Phi,d}(\Psi_\Phi(\mx{x}))] \circ [D\Psi_\Phi(\mx{x})] 
= [D\Psi_\Phi(d\mx{x})] \circ m_d
\]
(using the fact that $m_d$ is already linear). 
Therefore we shall determine when $\Psi_\Phi(\mx{x})$ 
is a critical point of $T_{\Phi,d}$. 

First suppose that $\mx{x}$ is not a critical point of 
$\Psi_\Phi$, i.e., $\mx{x} \notin \mathcal{H}_\Phi$. 
Then we can rewrite the above equation as 
\[
[DT_{\Phi,d}(\Psi_\Phi(\mx{x}))]
= [D\Psi_\Phi(d\mx{x})] \circ m_d 
  \circ [D\Psi_\Phi(\mx{x})]\inv\text.
\]
This equation implies that $\Psi_\Phi(\mx{x})$ is 
a critical point of $T_{\Phi,d}$ whenever $d\mx{x}$ 
is a critical point of $\Psi_\Phi$, i.e., when 
$d\mx{x} \in \mathcal{H}_\Phi$. 

Now the set of critical points is closed, 
and so every point in the closure of 
${m_d}\inv(\mathcal{H}_\Phi) \setminus \mathcal{H}_\Phi$ 
(which is to say, the union of all hyperplanes which are 
strict preimages of hyperplanes in $\mathcal{H}_\Phi$) 
also yields a critical point of $T_{\Phi,d}$. The critical 
values of $T_{\Phi,d}$ are therefore the images of 
$\mathcal{H}_\Phi$ by $\Psi_\Phi$, which is to say 
$\mathcal{D}_\Phi$.

Finally, note that $\mathcal{H}_\Phi$ is invariant under 
$m_d$, because $d\cdot H_{\mx{v},\ell} = 
H_{\mx{v},d\ell}$ and $d$ is an integer. Therefore all 
critical values of $T_{\Phi,d}$ lie in $\mathcal{D}_\Phi$, 
every point of $\mathcal{D}_\Phi$ is a critical value, 
and $\mathcal{D}_\Phi$ is invariant under $T_{\Phi,d}$.
\end{proof}

Recall that a covering map 
$p : \mathcal{Y} \to \mathcal{X}$ of 
path-connected topological spaces is called {\em regular} 
when the group of deck transformations 
$\mathrm{Gal}(\mathcal{Y}/\mathcal{X})$ 
acts transitively on each fiber of $p$. 

\begin{lemma}\label{L:galoisgroup}
Let $\Phi \subset \R^n$ be a root system having 
affine Weyl group $\til{W}_\Phi$. 
Set $\mathcal{X} = \C^n \setminus \mathcal{D}_\Phi$ 
and $\mathcal{Y} = \C^n \setminus \mathcal{H}_\Phi$. 
Then the restriction of the generalized cosine $\Psi_\Phi$ 
to $\mathcal{Y}$ is a regular covering of $\mathcal{X}$, 
with $\mathrm{Gal}(\mathcal{Y}/\mathcal{X}) = \til{W}_\Phi$.
\end{lemma}

\begin{proof}
By Lemma~\ref{L:Psi-critical}, no points of $\mathcal{Y}$ 
are critical for $\Psi_\Phi$, and therefore $\Psi_\Phi$ is 
locally a homeomorphism when restricted to $\mathcal{Y}$; 
i.e., $\Psi_\Phi\vert_{\mathcal{Y}}$ is a covering map. 
By definition, we have $\mathcal{D}_\Phi = 
\Psi_\Phi(\mathcal{H}_\Phi)$, so $\Psi_\Phi(\mathcal{Y}) 
= \mathcal{X}$. As observed in the proof of 
Lemma~\ref{L:Psi-critical}, 
$\Psi_\Phi(\til{w}\mx{x}) = \Psi_\Phi(\mx{x})$ for all 
$\til{w} \in \til{W}_\Phi$, so $\til{W}_\Phi$ is contained 
in $\mathrm{Gal}(\mathcal{Y}/\mathcal{X})$. Moreover, 
the same proof shows that the fiber over each point of 
$\mathcal{X}$ can be identified with $\til{W}_\Phi$, which 
implies $\mathrm{Gal}(\mathcal{Y}/\mathcal{X}) = \til{W}_\Phi$.
\end{proof}

An immediate consequence of Lemma~\ref{L:galoisgroup} 
is an expression for the fundamental group 
$\pi_1(\C^n\setminus\mathcal{D}_\Phi)$ as an extension 
of $\til{W}_\Phi$. Recall that any covering map 
$p : \mathcal{Y} \to \mathcal{X}$ induces an injective 
group homomorphism 
$p_* : \pi_1(\mathcal{Y}) \to \pi_1(\mathcal{X})$, 
defined by $p_*([\eta]) = [p \circ \eta]$. 
The subgroup $p_*(\pi_1(\mathcal{Y}))$ is normal in 
$\pi_1(\mathcal{X})$ precisely when $p$ is a regular 
covering, and in this situation the quotient 
$\pi_1(\mathcal{X})/p_*(\pi_1(\mathcal{Y}))$ 
is isomorphic to the deck transformation group 
$\mathrm{Gal}(\mathcal{Y}/\mathcal{X})$. 
(See \cite{aH} for details.)

\begin{corollary}
Given a root system $\Phi$ with affine Weyl group 
$\til{W}_\Phi$, let $\Psi_\Phi$, $\mathcal{H}_\Phi$, 
and $\mathcal{D}_\Phi$ be defined as above. Then we 
have the following exact sequence:
\[
\begin{CD}
0 @>>> \pi_1(\C^n\setminus\mathcal{H}_\Phi) 
  @>>> \pi_1(\C^n\setminus\mathcal{D}_\Phi) 
  @>>> \til{W}_\Phi @>>> 0
\end{CD}
\]
where the map 
$\pi_1(\C^n\setminus\mathcal{H}_\Phi) 
 \to \pi_1(\C^n\setminus\mathcal{D}_\Phi)$ is the 
injection $(\Psi_\Phi)_*$, and the map 
$\pi_1(\C^n\setminus\mathcal{D}_\Phi)\to\til{W}_\Phi$ 
is the induced canonical projection.
\end{corollary}

\section{Iterated monodromy groups}\label{S:img}

Iterated monodromy groups of dynamical systems 
(and of topological automata more generally) were 
introduced by V.~Nekrashevych \cite{vN-img,vN-ssg,vN-sym}. 
We recall the definition, using slightly different 
notation.

\begin{definition}[partial self-covering, monodromy action, iterated 
monodromy group]
Let $\mathcal{X}$ be a path-connected, locally path-connected 
topological space. A {\em partial self-covering} of $\mathcal{X}$ is a 
covering map $f : \mathcal{X}_1 \to \mathcal{X}$, where $\mathcal{X}_1$ 
is an open, path-connected subset of $\mathcal{X}$. Each iterate $f^k$ 
of a partial self-covering of $\mathcal{X}$ is again a partial 
self-covering, with domain $\mathcal{X}_k = f^{-k}(\mathcal{X})$. 
We will label a partial self-covering by the pair $(\mathcal{X},f)$.

Given a partial self-covering $(\mathcal{X},f)$ and a point 
$x_0 \in \mathcal{X}$, let $\mx{T}_f$ be the {\em tree of 
preimages of $x_0$}, namely, the vertex set of $\mx{T}_f$ 
is the disjoint union 
\[
\mx{T}_f = \bigsqcup_{k\ge0} f^{-k}(x_0),
\]
and $\mx{T}_f$ has an edge from $x' \in f^{-k}(x_0)$ 
to $x'' \in f^{-(k-1)}(x_0)$ if $x'' = f(x')$. If 
$f$ has topological degree $\delta$, then $\mx{T}_f$ 
is a rooted $\delta$-ary tree with root $x_0$. 

The fundamental group $\pi_1(\mathcal{X},x_0)$ acts on 
$\mx{T}_f$ as follows: given a loop $\gamma$ based at $x_0$ 
and $x' \in f^{-k}(x_0)$, use $f^k$ to lift $\gamma$ to a 
path $\tilde\gamma$ starting at $x'$, and let 
$[\gamma] \cdot x'$ be the endpoint of $\tilde\gamma$. 
This is the {\em monodromy action}, which 
induces the {\em monodromy homomorphism} 
$\mu_f : \pi_1(\mathcal{X},x_0) \to \Aut(\mx{T}_f)$,
\[
\mu_f([\gamma]) : x' \mapsto [\gamma] \cdot x'.
\]
The image of $\pi_1(\mathcal{X},x_0)$ via $\mu_f$ is the 
{\em iterated monodromy group} of $f$, denoted $\mathrm{IMG}(f)$. 
\end{definition}

It is not hard to check that, up to isomorphism, 
$\mathrm{IMG}(f)$ is independent of the choice of 
basepoint $x_0$. However, in what follows we will 
occasionally need to be attentive to basepoints 
for other reasons.

\begin{example}
Suppose $f : \C^n \to \C^n$ is post-critically finite, 
with critical locus $\mathcal{C}$ and post-critical 
locus $\mathcal{P}$. Then the restriction of $f$ to 
$\C^n \setminus (\mathcal{C}\cup\mathcal{P})$ is a 
partial self-covering of $\mathcal{X} = \C^n\setminus\mathcal{P}$. 
In this situation, we define $\mathrm{IMG}(f)$ to be 
the iterated monodromy group of $(\mathcal{X},f)$.
\end{example}

Given a partial self-covering $(\mathcal{X},f)$, 
it follows from the definitions that 
$[\gamma] \in \pi_1(\mathcal{X},x_0)$ is in the kernel 
of the monodromy homomorphism $\mu_f$ if and only 
if every lift of $\gamma$ by every iterate of $f$ 
is a loop (i.e., closed). This observation will be 
useful at several points.

\begin{example}[cf.~\cite{vN-img}]
Let $T_d : \C \to \C$ be the $d$th Chebyshev polynomial 
(as in Example~\ref{Ex:A1}), and set 
$\mathcal{X} = \C \setminus \{\pm2\}$. 
The $d-1$ critical points of $T_d$ are $2\cos(j\pi/d)$, 
$1 \le j \le d-1$, and the images of these points lie in 
$\{\pm2\}$; moreover, $\{\pm2\}$ is forward invariant under 
$T_d$. Thus the restriction of $T_d$ to $\mathcal{X}_1 = 
\C \setminus \{ 2\cos(j\pi/d) \mid 0 \le j \le d\}$ is a 
partial self-covering of $\mathcal{X}$. The fundamental 
group of $\mathcal{X}$ with basepoint $0$ is generated by 
$[\gamma_+]$ and $[\gamma_-]$, where $\gamma_\pm$ are the 
loops defined by $\gamma_\pm(s) = \pm2(1-e^{2\pi is})$ 
(Figure~\ref{F:fundgrp}, left). Using the 
relation $T_d\big(t+t\inv\big) = t^d+t^{-d}$, it can be 
seen that $[\gamma_+]$ and $[\gamma_-]$ both act on the 
tree $\mathbf{T}_{T_d}$ as order $2$ automorphisms 
(Figure~\ref{F:fundgrp}, right).
\begin{figure}[htb]
\centering
\begin{minipage}{2in}
\centering
\includegraphics{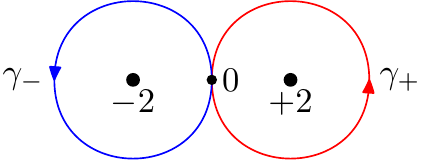}
\end{minipage}
\hspace{0.35in}
\begin{minipage}{2.2in}
\centering
\includegraphics[scale=1.35]{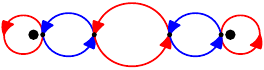}
\end{minipage}
\caption{{\sc Left:} Generators $\gamma_\pm$ of the fundamental 
group of $\C \setminus \{+2,-2\}$. {\sc Right:} Lifts of $\gamma_\pm$ 
by the Chebyshev polynomial $T_4 = T_2^2$. The large dots represent 
$-2$ and $+2$. Red curves are lifts of $\gamma_+$, and blue curves are 
lifts of $\gamma_-$. Each curve begins and ends at a point of 
$T_4\inv(0)$.}
\label{F:fundgrp}
\end{figure}
On the other hand, the 
product $[\gamma_-][\gamma_+]$ acts on the $k$th level of 
$\mathbf{T}_{T_d}$ as a permutation of order $d^k$; therefore 
the order of $\mu_{T_d}\big([\gamma_-][\gamma_+]\big)$ is 
infinite. Thus the iterated monodromy group of $T_d$ is 
isomorphic to the infinite dihedral group, or in other words 
the affine Weyl group of the $A_1$ root system.
\end{example}

\begin{definition}[semiconjugacy]
Two partial self-coverings $g : \mathcal{Y}_1 \to \mathcal{Y}$ 
and $f : \mathcal{X}_1 \to \mathcal{X}$ are {\em semiconjugate} 
if there exists a continuous map $p : \mathcal{Y} \to \mathcal{X}$ 
such that $p(\mathcal{Y}_1) = \mathcal{X}_1$ and 
$p \circ g = f \circ p$ on $\mathcal{Y}_1$. The map $p$ is then 
called a {\em semiconjugacy} from $g$ to $f$, and we write 
$p : (\mathcal{Y},g) \to (\mathcal{X},f)$.
\end{definition}

We are particularly interested in certain cases where 
two partial self-coverings $(\mathcal{X},f)$ and 
$(\mathcal{Y},g)$ are semiconjugate by a covering map 
$p : \mathcal{Y} \to \mathcal{X}$. 

\begin{lemma}\label{L:diagram}
Let $(\mathcal{X},f)$ and $(\mathcal{Y},g)$ be 
partial self-coverings, with 
$p : (\mathcal{Y},g) \to (\mathcal{X},f)$ a semiconjugacy. 
Suppose that $p$ is a regular covering map such that 
$p_*(\ker\mu_g) \subseteq \ker\mu_f$. Choose a basepoint 
$y_0 \in \mathcal{Y}$, and set $x_0 = p(y_0)$. Then the diagram
\begin{equation}\label{Eq:cd}
\begin{CD}
@. 0 @. 0 @. 0 @. \\
@. @VVV @VVV @VVV \\
0 @>>> \ker\mu_g @>>> \ker\mu_f @>>> \ker\mu_f/\ker\mu_g @>>> 0 \\
@. @VVV @VVV @VVV \\
0 @>>> \pi_1(\mathcal{Y},y_0) @>{p_*}>> \pi_1(\mathcal{X},x_0) @>>> \mathrm{Gal}(\mathcal{Y}/\mathcal{X}) @>>> 0 \\
@. @VV{\mu_g}V @VV{\mu_f}V @VVV \\
0 @>>> \mathrm{IMG}(g) @>>> \mathrm{IMG}(f) 
  @>>> \mathrm{IMG}(f)/\mathrm{IMG}(g) @>>> 0 \\
@. @VVV @VVV @VVV \\
@. 0 @. 0 @. 0 @. 
\end{CD}
\end{equation}
is commutative, with exact rows and columns. 
\end{lemma}

\begin{proof}
The vertical maps $\ker\mu_g \to \pi_1(\mathcal{Y},y_0)$ 
and $\ker\mu_f \to \pi_1(\mathcal{X},x_0)$ are inclusions, 
and so the upper left square of the diagram commutes by 
assumption. Outside of this square, the maps are all 
canonically determined, and the rest of the diagram 
commutes by standard theorems of group theory. Exactness 
is by construction.
\end{proof}

We close this section with a set of sufficient conditions 
for the inclusion $p_*(\ker\mu_g) \subseteq \ker\mu_f$ 
in the hypothesis of Lemma~\ref{L:diagram} to hold, which 
will allow us to employ the diagram \eqref{Eq:cd} in our 
arguments of the next section.

\begin{lemma}\label{L:kernelincl}
Let $p : (\mathcal{Y},g) \to (\mathcal{X},f)$ be a 
semiconjugacy of partial self-coverings. 
If $g^k$ is a regular covering map for all $k$, 
and $p$ is a regular covering map such that 
$p^{-1}(\mathcal{X}_k) = \mathcal{Y}_k$ for all $k$, 
then $p_*(\ker\mu_g) \subseteq \ker\mu_f$.
\end{lemma}

\begin{proof}
Choose a basepoint $y_0 \in \mathcal{Y}$, and set 
$x_0 = p(y_0)$. Let $[\eta] \in \ker\mu_g$; then every 
lift of $\eta$ by every iterate $g^k$ is a loop. 
Set $\gamma = p \circ \eta$, and let $\tilde\gamma$ be 
a lift of $\gamma$ by some iterate $f^k$. 
We want to show that $\tilde\gamma$ is a loop. 
By definition, $\tilde\gamma$ starts at some 
$x' \in f^{-k}(x_0) \subseteq \mathcal{X}_k$. 
Choose $y' \in p^{-1}(x') \subseteq \mathcal{Y}_k$, 
and lift $\tilde\gamma$ to a path in $\mathcal{Y}_k$ 
starting at $y'$. Now, $\tilde\eta$ is not necessarily 
a lift of $\eta$ by $g^k$. 
However, if we apply $p$ to $g^k(y')$, we find 
$p(g^k(y')) = f^k(p(y')) = f^k(x') = x_0$, and so there 
exists $\tau \in \mathrm{Gal}(\mathcal{Y}/\mathcal{X})$ 
such that $y_0 = \tau(g^k(y'_0))$, and $\tilde\eta$ is 
a lift of $\eta$ by $\tau\circ g^k$. 
Because $g^k$ is a regular covering map, the lifting 
property implies that there exists a homeomorphism 
$\tau' : \mathcal{Y}_k \to \mathcal{Y}_k$ such that 
$\tau\circ g^k = g^k \circ \tau'$. 
Thus $\tau' \circ \tilde\eta$ is a lift of $\eta$ by $g^k$, 
which means that $\tau' \circ \tilde\eta$ must be a loop. 
Because $\tau'$ is a homeomorphism, $\tilde\eta$ must also 
be a loop, and thus $p \circ \tilde\eta = \tilde\gamma$ 
is a loop. Therefore $[\gamma] = p_*([\eta]) \in \ker\mu_f$.
\end{proof}

It is worth making a couple of remarks on the 
condition $p^{-1}(\mathcal{X}_k) = \mathcal{Y}_k$ 
in the statement of Lemma~\ref{L:kernelincl}. 
First, for a general semiconjugacy 
$p : (\mathcal{Y},g) \to (\mathcal{X},f)$, 
we have only the inclusion 
$\mathcal{Y}_k \subseteq p^{-1}(\mathcal{X}_k)$. 
Second, when $p$ is a regular covering map, the equation 
$p^{-1}(\mathcal{X}_k) = \mathcal{Y}_k$ is equivalent 
to the statement that $\mathrm{Gal}(\mathcal{Y}/\mathcal{X})$ 
preserves $\mathcal{Y}_k$, in the sense that 
$\tau(\mathcal{Y}_k) = \mathcal{Y}_k$ for all 
$\tau \in \mathrm{Gal}(\mathcal{Y}/\mathcal{X})$.

\section{Proof of main theorem}\label{S:computation}

In Section~\ref{S:critical} we saw that all 
Chebyshev-like maps $T_{\Phi,d}$ are post-critically 
finite. Thus they have iterated monodromy groups, 
which we compute in this section.

\begin{theorem}\label{T:imgT}
Let $\Phi$ be a root system with affine Weyl group 
$\til{W}_\Phi$. For any $d \ge 2$, the iterated 
monodromy group of $T_{\Phi,d}$ is isomorphic to 
$\til{W}_\Phi$.
\end{theorem}

Before completing the proof of Theorem \ref{T:imgT}, 
we make one more general observation.

\begin{lemma}\label{L:trivialIMG}
Let $(\mathcal{X},f)$ be a partial self-covering. 
If $f$ is injective, then $\mathrm{IMG}(f) = 0$.
\end{lemma}
\begin{proof}
If $f$ is injective, then it is a homeomorphism. 
In this case, any lift of any loop by any iterate of $f$ 
remains a loop, and therefore all of $\pi_1(\mathcal{X})$ 
lies in the kernel of $\mu_f$. 

Alternatively, observe that when $f$ is injective, 
every level of $\mathbf{T}_f$ has only one vertex, and therefore 
$\pi_1(\mathcal{X})$ must act trivially at every level.
\end{proof}

We can apply Lemma~\ref{L:trivialIMG} to the 
partial self-covering $(\C^n\setminus\mathcal{H}_\Phi,m_d)$, 
because $m_d$ is evidently injective. Thus we have 
$\ker\mu_{m_d} = \pi_1(\C^n\setminus\mathcal{H}_\Phi)$.
The following lemma is now the primary piece that remains 
to be established.

\begin{lemma}\label{L:kernelsurj}
Let $\Phi$ be a root system, and let $\Psi_\Phi$ be its 
associated generalized cosine. Given $d \ge 2$, let 
$T_{\Phi,d}$ be the associated Chebyshev-like map. Then 
$(\Psi_\Phi)_*$ is a surjective map from 
$\pi_1(\C^n \setminus \mathcal{H}_\Phi)$ to 
$\ker\mu_{T_{\Phi,d}}$.
\end{lemma}

\begin{proof}
Let $\mathcal{Y} = \C^n \setminus \mathcal{H}_\Phi$, 
and choose a point $y_0 \in \R^n \cap \mathcal{Y}$. 
Set $x_0 = \Psi_\Phi(y_0)$.

First we check that the conditions of 
Lemma~\ref{L:kernelincl} are met, in order to see 
that $(\Psi_\Phi)_*(\pi_1(\mathcal{Y},y_0)) 
\subseteq \ker\mu_{T_\Phi,d}$. 
Because $g = m_d$ is a homeomorphism, $g^k$ is a 
regular covering for all $k$. 
The restriction of $\Psi_\Phi$ to $\mathcal{Y}$ 
is a regular covering (by Lemma~\ref{L:galoisgroup}), 
and so it suffices to check that 
$\mathcal{Y}_k = \C^n \setminus \frac{1}{d^k}\mathcal{H}_\Phi$ 
is invariant under $\mathrm{Gal}(\mathcal{Y}/\mathcal{X}) = 
\til{W}_\Phi$, which is true because $\frac{1}{d^k}\mathcal{H}_\Phi$ 
is even invariant under $\frac{1}{d^k}\til{W}_\Phi$, which 
contains $\til{W}_\Phi$.

Given $[\gamma] \in \pi_1(\C^n\setminus\mathcal{D}_\Phi,x_0)$, 
let $\eta$ be a lift of $\gamma$ by $\Psi_\Phi$ to a path 
in $\mathcal{Y}$. Then the endpoints of $\eta$ join points 
that differ by an element of $\til{W}_\Phi$. Suppose that 
$\til\gamma$ is a lift of $\gamma$ by $(T_{\Phi,d})^k$. 
Lift $\til\gamma$ to a path $\til\eta$ by $\Psi_\Phi$. 
Using the relation 
$(T_{\Phi,d})^k \circ \Psi_\Phi = \Psi_\Phi \circ m_{d^k}$, 
we see that $\til\eta = \frac{1}{d^k}\eta$, up to 
an element of $\til{W}_\Phi$. If $\til\gamma$ is also 
a loop, then the endpoints of $\til\eta$ must again 
differ by an element of $\til{W}_\Phi$. Thus, if 
$[\gamma] \in \ker\mu_{T_{\Phi,d}}$, it must be true 
that, for all $k$, the path $\frac{1}{d^k}\eta$ joins 
points that differ by some element of $\til{W}_\Phi$. 
We wish to show that this condition implies that 
$\eta$ is a closed loop.

Let $Q_\Phi^\vee \subset \R^n$ be the lattice 
generated by the coroots of $\Phi$. For each 
$a \in Q_\Phi^\vee$, the path $a + \eta$ also projects 
to $\gamma$. We can choose an element of $Q_\Phi^\vee$ 
that sends the endpoints of $\eta$ to a single Weyl 
chamber of $\Phi$. (For instance, we can assume that 
the endpoints of $\eta$ are linear combinations of 
simple roots with positive coefficients.) The elements 
of $\mathcal{H}_\Phi$ partition this Weyl chamber into 
regions of finite area, each of which is a fundamental 
domain for $\Psi_\Phi(\R^n)$. Now we can find $k$ 
sufficiently large that both endpoints of 
$\frac{1}{d^k}\eta$ are in a single fundamental domain. 
This is impossible unless $\eta$ is a loop, which proves 
the result.
\end{proof}

\begin{proof}[Proof of Theorem \ref{T:imgT}]
Consider the diagram \eqref{Eq:cd}, with 
$\mathcal{X} = \C^n \setminus \mathcal{D}_\Phi$, 
$\mathcal{Y} = \C^n \setminus \mathcal{H}_\Phi$, 
$f = T_{\Phi,d}$, $g = m_d$, and $p = \Psi_{\Phi}$. 
By Lemma~\ref{L:trivialIMG}, $\mathrm{IMG}(m_d)$ is trivial, 
which implies that $\ker \mu_{m_d} = \pi_1(\mathcal{Y},y_0)$, 
and also that $\mathrm{IMG}(T_{\Phi,d})/\mathrm{IMG}(m_d) 
= \mathrm{IMG}(T_{\Phi,d})$. 
Then Lemma~\ref{L:kernelsurj} implies that 
$(\Psi_\Phi)_* : \ker\mu_{m_d} \to \ker\mu_{T_{\Phi,d}}$ 
is an isomorphism, so 
$\ker\mu_{T_{\Phi,d}}/\ker\mu_{m_d} = 0$. The exactness 
of the rows and columns of \eqref{Eq:cd} now shows that 
$\mathrm{IMG}(T_{\Phi,d}) 
\cong \mathrm{Gal}(\mathcal{Y}/\mathcal{X})$, 
which by Lemma~\ref{L:galoisgroup} is precisely 
the affine Weyl group of $\Phi$.
\end{proof}

Finally, we state a consequence for the structure of 
affine Weyl groups, for which we need one more set of 
definitions (cf. \cite{sG12,vN-img,vN-ssg}).

\begin{definition}[$\delta$-ary tree, self-similar group]
Given a positive integer $\delta$, the 
\emph{$\delta$-ary tree} is the graph $\mathbf{T}_\delta$ 
whose vertex set consists of all finite words in the alphabet 
$[\delta] = \{1,2,\dots,\delta\}$, with an edge between each 
pair of vertices $w$ and $wk$, where $k \in [\delta]$. 
The \emph{root} of $\mathbf{T}_\delta$ is the empty word 
$\varnothing$. 
For each $k \in [\delta]$, the subtree 
$\mathbf{T}_{\delta,k}$ is the induced graph on the 
set of vertices that begin with $k$. The map 
$\sigma_k : w \mapsto kw$ is an isomorphism from 
$\mathbf{T}_\delta$ to $\mathbf{T}_{\delta,k}$. 
Given an automorphism $g$ of $\mathbf{T}_\delta$ and 
$k \in [\delta]$, the \emph{renormalization 
of $g$ at $k$} is the induced automorphism $g_k$ of 
$\mathbf{T}_\delta$ defined by 
$g_k = \sigma_{g(k)}^{-1} \circ g \circ \sigma_k$. 
We say that a group $G$ of automorphisms of 
$\mathbf{T}_\delta$ is \emph{self-similar} if 
$g_k \in G$ for all $g \in G$ and for all 
$k \in [\delta]$. 
\end{definition}

If $(\mathcal{X},f)$ is any partial self-covering having 
topological degree $\delta$, then the tree of preimages 
$\mathbf{T}_f$ can be identified with $\mathcal{T}_\delta$ 
in a canonical (but non-unique) way, and under this 
identification $\mathrm{IMG}(f)$ is a self-similar group 
acting faithfully on $\mathcal{T}_\delta$. The construction 
of the Chebyshev-like map $T_{\Phi,d}$ from a root system 
of rank $n$ implies that the topological degree of 
$T_{\Phi,d}$ is $d^n$, which leads to the following result.

\begin{corollary}
Let $\Phi$ be a root system of rank $n$. Then, 
for any $d \ge 2$, $\til{W}_\Phi$ acts faithfully 
as a self-similar group on the $d^n$-ary tree as 
the iterated monodromy group $\mathrm{IMG}(T_{\Phi,d})$.
\end{corollary}

\subsection*{Acknowledgments}

I thank Jim Belk and Roland Roeder for helpful 
conversations. I am also grateful to an anonymous 
referee for pointing out an error in an earlier 
version of Lemma~\ref{L:diagram} and its proof.

\end{document}